\documentclass[a4paper,12pt]{article}
\usepackage[english]{babel}
\usepackage[utf8]{inputenc}
\usepackage{paralist,url,verbatim, anysize}
\usepackage{amscd}
\usepackage[all,cmtip]{xy}
\usepackage{amsmath,amsthm,amssymb,amsfonts}
\usepackage[bookmarksopen=true]{hyperref}
\usepackage{color,graphicx}

\theoremstyle{plain}
\newtheorem{theorem}{\bf Theorem}[section]
\newtheorem{conjecture}[theorem]{Conjecture}
\newtheorem{fconjecture}[theorem]{(Disproved) Conjecture}
\newtheorem{corollary}[theorem]{Corollary}
\newtheorem{lemma}[theorem]{Lemma}
\newtheorem{proposition}[theorem]{Proposition}
\newtheorem{mthm}{\bf Main Theorem}
\newtheorem{problem}[theorem]{Problem}

\theoremstyle{definition}
\newtheorem{remark}[theorem]{Remark}

\newtheorem{example}[theorem]{Example}

\newcommand{\codim}{\operatorname{codim} }

\newcommand{\N}{\mathbb{N}}
\newcommand{\Z}{\mathbb{Z}}
\newcommand{\PP}{\mathbb{P}}

\newcommand{\CC}{\mathbb{C}}
\newcommand{\KK}{\mathbb{K}}

\newcommand{\diam}{\operatorname{diam} }
\newcommand{\Min}{\operatorname{Min} }
\newcommand{\height}{\operatorname{height} }
\newcommand{\Ker}{\operatorname{Ker} }
\newcommand{\Proj}{\operatorname{Proj} }

\newcommand{\pp}{\mathfrak{p}}
\newcommand{\qq}{\mathfrak{q}}
\newcommand{\mm}{\mathfrak{m}}
\renewcommand{\O}{\mathcal{O}}
\newcommand{\Hom}{\operatorname{Hom} }
\newcommand{\reg}{\operatorname{reg} }

\begin{document}

\title{On the dual graphs of Cohen--Macaulay algebras}

\author{
Bruno Benedetti \thanks{Supported by the DFG Collaborative Research Center TRR109, ``Discretization in Geometry and Dynamics''.}\\
\small Dept.\ Mathematik u.\ Informatik, FU Berlin\\
\small \url{bruno@zedat.fu-berlin.de}
\and 
Matteo Varbaro \thanks{Supported by PRIN  2010S47ARA\_003 ``Geometria delle Variet\`a Algebriche".} \\
\small Dip.\ Matematica, U. Genova \\
\small \url{varbaro@dima.unige.it}
}

\date{\today}
\maketitle 
\begin{abstract}
Given an equidimensional algebraic set $X\subset \PP^n$, its dual graph $G(X)$ is the graph whose vertices are the irreducible components of $X$ and whose edges connect components that intersect in codimension one. 
Hartshorne's connectedness theorem says that if (the coordinate ring of) $X$ is Cohen-Macaulay, then $G(X)$ is connected. We present two quantitative variants of Hartshorne's result:

\begin{compactenum}[(1)]
\item If $X$ is a Gorenstein subspace arrangement, then $G(X)$ is $r$-connected, where $r$ is the Castelnuovo--Mumford regularity of $X$. \\
(The bound is best possible. For coordinate arrangements, it yields an algebraic extension of Balinski's theorem for simplicial polytopes.)

\item If $X$ is an arrangement of lines no three of which meet in the same point, and $X$ is canonically embedded in $\PP^n$, then  
the diameter of the graph $G(X)$ is less than or equal to $\codim_{\PP^n}X$. \\
(The bound is sharp; for coordinate arrangements, it yields an algebraic expansion on the recent combinatorial result that the Hirsch conjecture holds for flag normal simplicial complexes.) 
\end{compactenum}
On the way to these results, we show that there exists a graph which is not the dual graph of any simplicial complex (no matter the dimension). 
\end{abstract}

\section{Introduction}
Let $I$ be an ideal in the polynomial ring $S=\KK[x_1, \ldots, x_n]$, where $\KK$ is some field. For simplicity, we assume throughout this paper that
$I$ is \emph{height-unmixed}, that is, all minimal primes of $I$ have the same height. The \textbf{dual graph} $G(I)$ is then naturally defined as follows: First we draw vertices $v_1, \ldots, v_s$, corresponding to the minimal prime ideals $\{\pp_1, \ldots, \pp_s\}$ of $I$. Then we connect two vertices $v_i$ and $v_j$ with an edge if and only if 
\[\height I = \;\height (\pp_i + \pp_j) - 1.\]

The dual graph need not be connected, as shown for example by the ideal $I =(x,y) \cap (z,w)$ inside $\CC[x,y,z,w]$, whose dual graph consists of two disjoint vertices. The reader familiar with combinatorics should note that this ideal is monomial and squarefree, so via the Stanley--Reisner correspondence it can be viewed as a simplicial complex. There is already an established notion of ``dual graph of a (pure) simplicial complex'' and it is compatible with our definition, in the sense that if $I_{\Delta}$ is the Stanley--Reisner ideal of a pure complex $\Delta$, the dual graphs of $\Delta$ and of $I_\Delta$ are the same. This way it is usually easy to produce examples of ideals with prescribed dual graphs. However, not all graphs are dual graphs of a simplicial complex, as we will see in Corollary \ref{cor:notallgraphsaredual}.


Having connected dual graph is a property well studied in the literature under the name of ``\emph{connectedness in codimension one}''. Remarkably, it is shared by all Cohen--Macaulay algebras:

\begin{theorem}[Hartshorne \cite{Ha}]  For any ideal $I \subset S$, if $S/I$ is Cohen--Mac\-aulay then $G(I)$ is connected. 
\end{theorem}

(It is well known that, if $S/I$ is Cohen-Macaulay, then $I$ is height-unmixed). 
But can we say more about how connected $G(I)$ is, if we know more about $I$ --- for example, that $I$ is generated in certain degrees, or that $S/I$ is Gorenstein? This leads to the following question.

\begin{problem}
Give a quantitative version of Hartshorne's connectedness theorem.
\end{problem}

There are at least two natural directions to explore: (a) lower bounds for the connectivity, and (b) upper bounds for the diameter. 

\medskip
 
\textbf{Connectivity} counts \emph{how many vertex-disjoint} paths there are (at least) between two arbitrary points of the graph. Balinski's theorem says that the graph of every $d$-polytope is $d$-connected. Since the dual graph of any $d$-polytope $P$ is also the $1$-skeleton of a $d$-polytope (namely, of the polar polytope $P^*$), an equivalent reformulation of Balinski's theorem is ``the dual graph of every $d$-polytope $P$ is $d$-connected''.
 This was later extended by many authors, cf.\ e.g.\ \cite{Barnette} \cite{Athanasiadis} \cite{Wotzlaw} \cite{BjVo}. Here is one extension due to Klee:

\begin{theorem}[Klee~\cite{Klee}]\label{ref:klee}
Let $I$ be the Stanley--Reisner ideal of a $d$-dimensional triangulated homology manifold (or more generally, of any $d$-dimensional normal pseudomanifold without boundary). The dual graph of $I$ is $(d+1)$-connected. 
\end{theorem}

Stanley--Reisner rings of homology spheres are particular examples of Gorenstein rings, so one can ask whether $S/I$ Gorenstein implies that $G(I)$ is highly connected. The answer is negative: As we show in Example \ref{ex:GorNot2conn},  
there are complete intersection ideals $I$ such that $G(I)$ is not even $2$-connected, because it has a leaf. 

Nevertheless, it is indeed possible to ``compromise'' between Hartshorne's theorem and Balinski and Klee's results. 
Recall that a radical ideal is said to \emph{define a subspace arrangement} if it is a finite intersection of (prime) ideals generated by linear forms.  

\begin{mthm}[Theorem \ref{thm:GorensteinRConnected}] \label{mthm:1}
Let $I\subset S$ be an ideal defining a subspace arrangement. If $S/I$ is Gorenstein and has Castelnuovo-Mumford regularity $r$, then $G(I)$ is $r$-connected.
\end{mthm}

The Stanley-Reisner ring of a simplicial (homology) $d$-sphere has Castelnuovo--Mum{\-}ford regularity $d+1$. So Main Theorem \ref{mthm:1} does imply that the dual graph of every (homology) $d$-sphere is $(d+1)$-connected. However, Main Theorem \ref{mthm:1} is much more general. In fact, the arrangements corresponding to squarefree monomial ideals are called \emph{coordinate}. Let $\mathfrak{L}$ be the class of all subspace arrangements obtainable from coordinate ones via linear changes of variables or via hyperplane sections; let $\mathfrak{P}$ be the class of subspace arrangements whose defining ideal is generated by a product of variables. It is well known that
\[
\{ \textrm{coordinate subspace arrangements}
\}
\; \subset \;
\mathfrak{L}
\; \subset \;
\mathfrak{P}
\; \subset \;
\{\textrm{all subspace arrangements} \},
\]
and most subspace arrangements are not in $\mathfrak{P}$, as explained in \cite{BPS}. 

Our proof of Main Theorem \ref{mthm:1} uses liaison theory, cf.\ \cite{Mi}, and a homological result by Derksen--Sidman \cite{DS}. The bound is best possible, in the sense that:
\begin{compactenum}[(1)]
\item The conclusion ``$r$-connected'' cannot be replaced by ``$(r+1)$-connected'' in general, cf. Example \ref{ex:RegBestPossible}.
\item The assumption ``$S/I$ Gorenstein'' cannot be weakened, for example, to ``$S/I$ Cohen--Macaulay'': See Remark \ref{rem:RegBestPossible}.
\item Without assuming that $I$ defines a subspace arrangement, the best one can prove is that $G(I)$ is 2-connected, provided the quotient of $S$ by any primary component of $I$ is Cohen--Macaulay (Corollary \ref{cor:2-conngor}). Without the latter assumption, one can infer nothing more than the connectedness of $G(I)$, even if $I$ is a complete intersection. Compare Example \ref{ex:GorNot2conn}.
\item Non-radical complete intersections whose radical defines a subspace arrangement, might have a path as dual graph, even if the regularity of $S/I$ is very high: See Example \ref{ex:unbounded}.
\end{compactenum}
\vskip5mm

The other direction in which Hartshorne's theorem could be extended, is by estimating the \textbf{diameter}. 
Recall that the diameter of a graph is defined as the maximal distance of two of its vertices; so connectedness is the same as having finite diameter.
But is there a sharp bound on $\diam G(I)$ depending only on the degree of the generators of $I$, say?

One result of this type has been recently found in the case of (squarefree) monomial ideals,  using ideas from metric geometry.

\begin{theorem}[Adiprasito--Benedetti \cite{AB}, cf. Section \ref{sec:combcomm}] \label{thm:AB}
Let $I \subset S$ be a monomial ideal generated in degree $2$. If $S/I$ is Cohen--Macaulay, then $\diam G(I) \le \height I$. 
\end{theorem}

Beyond the world of monomial ideals, however, the situation is much less clear. From now on, we will call \emph{Hirsch} the ideals $I$ such that $\diam G(I) \le \height I$. The name is inspired by a long-standing combinatorial problem, posed in 1957 by Warren Hirsch and recently solved in the negative by Santos \cite{Santos}, which can be stated as follows: 

\begin{fconjecture}[Hirsch]
If $\Delta$ is the boundary of a convex polytope, then $I_\Delta$ is Hirsch. 
 \end{fconjecture}

The work by Santos and coauthors \cite{MSW} implies that for any $k$ one can construct squarefree monomial ideals $I=I(k)$ with $S/I$ even Gorenstein, such that  $\diam G(I)=21k$ and $\height I=20k$ (Example \ref{ex:nonHirsch}). However, these non-Hirsch ideals are generated in high degree. This motivated us to make the following conjecture:

\begin{conjecture} \label{con:main}
Let $I \subset S$ be an arbitrary ideal generated in degree $2$. If $S/I$ is Cohen--Macaulay, then $I$ is Hirsch. 
\end{conjecture}

In Section \ref{sec:examples}, we show some partial argument in favor of Conjecture \ref{con:main}, proving it for all ideals of small height or regularity. A positive solution of Conjecture \ref{con:main} would instantly imply also a polynomial upper bound (in terms of the number of variables) for ideals generated in higher degree: See Proposition \ref{prop:RedToQuadrics}.

Using techniques that are essentially combinatorial, although some algebraic geometry is required for the setup, in Section \ref{sec:arrangement} we are able to obtain the following result:

\begin{mthm}[Theorem \ref{thm:LinesArrangement}]
Let $C\subset \PP^N$ be an arrangement of projective lines such that no three lines meet in the same point. If $C$ is canonically embedded, then its defining ideal $I$ is Hirsch, that is, the diameter of the graph $G(I)$ is not larger than $\codim_{\PP^N}C$. 
\end{mthm}

``Canonically embedded'' refers here to the technical requirement that the canonical sheaf $\omega_C$ is isomorphic to the pull-back of the twisted structural sheaf $\mathcal{O}_{\PP^N}(1)$. This condition is natural in order to produce embeddings that are quadratic and Cohen--Macaulay. (As a scheme, $C$ can be embedded in several ways; the canonical embedding tends to be quadratic, while other embeddings may result in ideals generated in very high degree.) 

\bigskip

The paper is structured as follows: we start with a ``background" section, consisting of essentially known results and useful reductions. The reader already familiar with combinatorics and commutative algebra may skip to Sections \ref{sec:gor} and \ref{sec:arrangement}, which form the core of the paper. Section \ref{sec:examples} is finally a repertoire of interesting examples.

\bigskip

The authors would like to thank Manolis Tsakiris for his comments concerning the previous version of this paper.

\section{Background}\label{sec:background}

\subsection{Combinatorics: Graph Connectivity and Diameter}
All graphs we consider have neither loops nor parallel edges.
A graph $G$ is called {\it $k$-vertex-connected} (or simply {\it $k$-connected}) if it has at least $k+1$ vertices, and any two vertices of $G$ are joined by at least $k$ vertex-disjoint paths. 
So $1$-connected is the same as connected. Similarly, $G$ is called {\it $k$-edge-connected} if it has at least $k+1$ vertices, and any two vertices of $G$ are joined by at least $k$ edge-disjoint paths. $1$-edge-connected is the same as connected. Obviously $k$-vertex-connected implies  $k$-edge-connected for all $k$. The converse is true only for $k=1$: for example, two squares glued together at a vertex yield a $2$-edge-connected graph that is not $2$-connected. In any $k$-edge-connected graph, every vertex has degree at least $k$. The converse is false.

There is a well known characterization of the two notions of connectivity:

\begin{theorem}[Menger] \label{thm:menger} 
Let $G$ be a graph on $n$ vertices. Let $0<k<n$ be an integer.
\begin{compactenum}[\rm(i)]
\item $G$ is $k$-connected $\iff$ $G$ cannot be disconnected by removing less than $k$ vertices, however chosen.
\item $G$ is $k$-edge-connected $\iff$ $G$ cannot be disconnected by removing less than $k$ edges, however chosen.
\end{compactenum}
\end{theorem}

For a direct proof of this, see [Diestel]; both (i) and (ii) are easy instances of Ford--Fulkerson's ``max-flow-min-cut theorem'', cf. \cite{Bollobas}.

The \emph{distance} of two vertices in a graph is the number of edges of a shortest path joining them. The \emph{diameter} of a graph is the maximum of the distances between its vertices. As the intuition suggests, the more connected a graph is, the shorter its diameter:

\begin{lemma}[folklore] \label{lem:menger} Let $G$ be a graph on $s$ vertices having $t$ edges;
\begin{compactenum}[\rm(a)]
\item if $G$ is $k$-connected, then $\diam G\le \lfloor (s-2)/k \rfloor + 1$; 
\item if $G$ is $k$-edge-connected, then $\diam G\le \lfloor t/k \rfloor$.
\end{compactenum}
\end{lemma}

\begin{proof} 
We show item (a); item (b) is analogous. Let $d$ be the diameter of $G$. 
If $d \le 1$ the claim is obvious. If $d \ge 2$, choose two vertices $x$, $y$ at distance $d$. By the connectivity assumption, there are $k$ vertex-disjoint paths joining $x$ and $y$. Each of these paths contains at least $d-1$ vertices in its relative interior. 
Together with $x$ and $y$, this yields a set of at least $k (d-1) + 2$ vertices inside $G$. So $k(d -1) + 2 \le  s$, whence the conclusion follows because $d=\diam G$ is an integer.
\end{proof}

For any connected graph $G$ with $s$ vertices, one has $\diam G \le s - 1$, with equality if and only if $G$ is a path. Since we are interested in upper bounds for the diameter, in the next section we review the known upper bounds on the number of vertices of $G=G(I)$.

\subsection{Commutative Algebra: The number of minimal primes}
Throughout this section, $S$ will denote the polynomial ring $\KK[x_1, \ldots, x_n]$; $I$ will be a height-unmixed graded ideal of $S$;  
$\Min(I)$ will denote the set of minimal primes of $I$. 

To provide an upper bound for the number of vertices of $G(I)$, let us recall a simple definition. 
If $d$ is the Krull dimension of $S/I$ there is a polynomial $h \in \Z[t]$, called the \emph{$h$-polynomial}, such that
$\sum_{i \in N} \, \dim_{\KK} \, (S/I)_i \; t^i \, \, = \, \frac{h(t)}{(1-t)^d}$.
The integer $e(S/I)=h(1)$ obtained by evaluating the $h$-polynomial at $1$ is called \emph{multiplicity}\footnote{The multiplicity is sometimes called \emph{degree} in the literature. We refrain from this notation to avoid confusions with the  degree of the polynomials generating $I$.}  of $S/I$. 
The multiplicity satisfies the following additive formula:

\begin{equation}\label{eq:addgen}
e(S/I)=\sum_{\pp\in \Min(I)}\mathrm{length} (S/I)_{\pp} \, \cdot \, e(S/\pp).
\end{equation}

From \eqref{eq:addgen} we see that $e(S/I)$ is a sum of $|\Min(I)|$ positive integers. This implies the following:

\begin{lemma}\label{lem:generalbound}
For any height-unmixed graded ideal $I$, the number of vertices of $G(I)$ is at most $e(S/I)$.
\end{lemma}

In case $I$ is a radical ideal, we have $I=\bigcap_{\pp\in\Min(I)}\pp$ and
$IS_{\pp}=\pp S_{\pp}$ for all $\pp\in\Min(I)$. In particular, $\mathrm{length} (S/I)_{\pp}=1$, which allows us to simplify Equation \eqref{eq:addgen} as follows:

\begin{equation}\label{eq:addrad}
e(S/I)=\sum_{\pp\in \Min(I)} \!\! e(S/\pp).
\end{equation}

In fact, if $I$ has no embedded primes, \eqref{eq:addrad} holds if and only if $I$ is radical.

\begin{remark} \label{rem:vgsa}
It is well known that $e(S/\pp)=1$ if and only if $\pp$ is generated by linear forms. So if $I\subset S$ is an ideal defining a subspace arrangement, 
\[e(S/I) \; = \; |\Min(I)| \; = \; \textrm{ number of vertices of } G(I).\]
 \end{remark}

Remark \ref{rem:vgsa} suggests that the case of subspace arrangements is one of the most promising for finding examples of ideals with large diameter.  For subspace arrangements, in fact, the graph $G(I)$ has the largest possible number $s$ of vertices -- so the upper bound $\diam G(I) \le s - 1$ becomes less restrictive.

To prove further upper bounds for the number of vertices of $G(I)$, we need to recall a classical definition. Let

\[\cdots \rightarrow F_j \rightarrow \cdots \rightarrow F_0 \rightarrow S/I \rightarrow 0\]
be a minimal graded free resolution for the quotient $S/I$. The \emph{Castelnuovo--Mumford regularity}  
$\reg(S/I)$ of $S/I$ is the smallest integer $r$ such that for each $j$, all minimal generators of $F_j$ have degree $\le r + j$. The regularity does not change if we quotient out by a regular element of degree one. It can be characterized using Grothendieck duality as follows:
\begin{equation} \label{eq:grothendieck}
\reg(S/I) = \max
\{ i+j \, : \, 
H^i_\mm  (S/I)_j \ne 0
\},
\end{equation}  
where $H^i_\mm$ stands for local cohomology with support in the maximal ideal $\mm=(x_1, \ldots, x_n)$. This implies the following, well-known lemma (cf. \cite[Theorem 4.4.3]{BrunsHerzog}):

\begin{lemma} \label{lem:deghreg}
Let $I$ be a graded ideal. Let $h(t)$ be the $h$-polynomial of $S/I$. If $S/I$ is Cohen--Macaulay, then 
$\deg(h)=\operatorname{reg} (S/I)$.
\end{lemma}

\begin{lemma} \label{lem:regularity}
Let $I\subset S$ be a height-unmixed graded ideal of height $c$. Let $s$ be the number of vertices of $G(I)$. 
\begin{compactenum}[\rm (i)]
\item If all minimal generators of $I$ have degree $\le k$, then
$s \leq k^c$. 
\item If $S/I$ is Cohen--Macaulay and has Castelnuovo-Mumford regularity $r$, then
\[s \leq \sum_{i=0}^{r} \binom{c+i-1}{i}.\]
\end{compactenum}
\end{lemma}

\begin{proof}
\begin{compactenum}[\rm (i)]
\item By Lemma \ref{lem:generalbound}, it suffices to prove that $e(S/I) \leq k^c$. Since the Hilbert function is preserved under field extensions, without loss of generality we may assume that $\KK$ is infinite. Let us choose an $S$-regular sequence $f_1,\ldots ,f_c$ of degree-$k$ polynomials such that $J=(f_1,\ldots ,f_c)\subset I$. Then
$e(S/I) \le e(S/J)=k^c$.
\item As before, we may assume that $\KK$ is infinite. A linear Artinian reduction of $S/I$ (that is $S/I$ mod out a linear system of parameters) will look like a $\KK$-vector subspace of $A=\displaystyle\frac{\KK[x_1,\ldots ,x_c]}{(x_1,\ldots ,x_c)^{r+1}}$. But then
$e(S/I) \, \le \, e(A) \, = \, \sum_{i=0}^{r} \binom{c+i-1}{i}$. \qedhere
\end{compactenum}
\end{proof}
 
 
 \subsection{Combinatorial commutative algebra: reduction to radicals} \label{sec:combcomm}
In this section we provide a quick overview on the situation for monomial ideals, showing that for such ideals the connectivity and diameter problems can be reduced to the radical case and ultimately to the world of simplicial complexes, where we can exploit the recent results of \cite{AB}. We sketch the basic definitions, referring to \cite[Chapter~1]{MiSt} for details. 

Let $n$ be a positive integer. A \emph{simplicial complex on $n$ vertices} is a finite collection $\Delta$ of subsets of $\{1, \ldots, n\}$ (called \emph{faces}) that is closed under taking subsets. The \emph{dimension} of a face is its cardinality minus one. A \emph{facet} is an inclusion-maximal face; ``$d$-face'' is short for ``$d$-dimensional face'' and ``vertex'' is short for ``$0$-face''. The dimension of a simplicial complex is the largest dimension of a face in it. A simplicial complex is \emph{pure} if all its facets have the same dimension. The \emph{dual graph} of a pure simplicial complex $\Delta$ is defined as follows: The graph vertices correspond to the facets of $\Delta$, and two vertices are connected by an edge if and only if the corresponding facets share a face of dimension one less. 

The \emph{Stanley--Reisner ideal} $I_\Delta$ of a simplicial complex $\Delta$ with $n$ vertices is the ideal of $\KK[x_1, \ldots ,x_n]$ defined by
$I_\Delta := \left(x_{i_1} \cdots x_{i_r} \; : \; \{i_1, \ldots, i_r\} \notin \Delta \right).$
By construction,  $I_\Delta$ is generated by squarefree monomials. Conversely, every radical monomial ideal $J$ is generated by squarefree monomials and can be written as $J=I_{\Delta}$ for a suitable simplicial complex~$\Delta$. So ``simplicial complexes on $n$ vertices'' are in bijection with ``radical monomial ideals of $S=\KK[x_1, \ldots ,x_n]$''. Moreover, the minimal primes of $I_\Delta$ can be described combinatorially  via the formula 
\[I_\Delta = \bigcap_{F \textrm{ facet of } \Delta}(x_i \, : i \notin F).\] 
 The height of an ideal generated by $c$ distinct variables is $c$. In particular, if $\Delta$ has $n$ vertices and all its facets are $d$-dimensional, the height of any minimal prime of $I_\Delta$ is $n-d-1$.

\begin{lemma} \label{lem:combcomm} If $I_{\Delta}$ is the Stanley--Reisner ideal of a pure simplicial complex $\Delta$, the dual graph of $\Delta$ is $G(I_\Delta)$.
\end{lemma} 

\begin{proof} Let $F, F'$ be two facets of $\Delta$. 
$F$ and $F'$ are adjacent in $\Delta$ if and only if $P_F$ and $P_{F'}$ have the same monomial generators, except one; if and only if 
$\height(I_{\Delta}) = \height(P_F + P_{F'}) - 1$; if and only if $P_F$ and $P_{F'}$ are adjacent in $G(I_{\Delta})$. 
\end{proof}

A simplicial complex is called \emph{flag} if the Stanley-Reisner ideal of the complex is generated in degree two. A simplicial complex $\Delta$ is called \emph{Cohen--Macaulay} (over $\KK$) if $\KK[x_1, \ldots ,x_n] / I_{\Delta}$ is Cohen--Macaulay. A simplicial complex is called \emph{strongly connected} if its dual graph is connected. The \emph{star} of a face $F$ in a simplicial complex $C$ is the smallest subcomplex containing all faces of $C$ that contain $F$. A simplicial complex is called \emph{normal} if it is strongly connected, and so are the stars of all its faces. It is well known that Cohen--Macaulay complexes are normal. 

Let $\Delta$ be a $d$-dimensional simplicial complex, and let $F, G$ be two adjacent $d$-simplices. Let $v$ be the only vertex that belongs to $F$ but not to $G$. When we move from $F$ to $G$, we abandon the star of $v$ in $\Delta$.
A path in the dual graph of $\Delta$ is called \emph{non-revisiting} if it never reenters the star of a vertex previously abandoned. It is easy to see  that in a $d$-dimensional simplicial complex with $n$ vertices, any non-revisiting dual path can be at most $n-d-1$ steps long. These notions are interesting for our diameter problem because of the following recent result:
  
\begin{theorem}[{Adiprasito--Benedetti \cite{AB}}] \label{thm:monomialOriginal}
Let $\Delta$ be a flag normal simplicial complex of dimension $d$ and with $n$ vertices. Then any two facets of $\Delta$ can be connected via a non-revisiting path. In particular, the diameter of the dual graph of $\Delta$ is $\le n-d-1$. 
\end{theorem}

The proof uses ideas of metric geometry applied to simplicial complexes. Below we present an algebraic consequence.

\begin{lemma}[cf.~{\cite{HTT}}] \label{lem:HTT} Let $I$ be an ideal of $S=\KK[x_1, \ldots , x_n]$.  Let $\sqrt{I}$ be the radical of $I$. If $S/I$ is Cohen-Macaulay, $S/\sqrt{I}$ need not be Cohen--Macaulay. However, if $I$ is monomial and $S/I$ is Cohen-Macaulay, so is $S/\sqrt{I}$.
\end{lemma}

\begin{corollary} \label{cor:monomial}
Let $I$ be a monomial ideal such that $S/I$ is Cohen-Macaulay. If $I$ is generated in degree $2$ (or more generally, if each minimal generator has a support of~$\le 2$ variables), then $\diam G(I) \le \height I$. 
\end{corollary}

\begin{proof} 
Clearly, also $\sqrt{I}$ is generated in degree at most $2$; moreover, $\height \sqrt{I}=\height I$ and $G(\sqrt{I})=G(I)$. Furthermore, $S/\sqrt{I}$ is Cohen-Macaulay by Lemma \ref{lem:HTT}. Since $\sqrt{I}$ is radical and monomial, it is the Stanley--Reisner ring of some simplicial complex $\Delta$. By the assumptions, $\Delta$ is flag and Cohen-Macaulay, so in particular normal. 
Moreover, if $\Delta$ has dimension $d$ and $n$ vertices, by Theorem \ref{thm:monomialOriginal} the dual graph of $\Delta$ has diameter $\le n-d-1$. 
Since $\height \sqrt{I} = n - d - 1$, via Lemma \ref{lem:combcomm} we conclude  \[\diam G(I) = \diam G(\sqrt{I}) \le n-d-1 = \height \sqrt{I} = \height I. \qedhere\] 
\end{proof}

\subsection{Reduction to quadrics} \label{sec:ReductionToQuadrics}
Here we show that ideals generated in degree 2 play a special role in understanding dual graphs of Cohen-Macaulay projective algebraic objects. In fact, there is a classical algebraic procedure, named after Giuseppe Veronese, that allows us to associate any Cohen--Macaulay algebra with a Cohen--Macaulay \emph{quadratic} algebra with the same dual graph.

Let $k, d, n$ be positive integers. Let $I$ be an ideal of $S=\KK[x_1,\ldots ,x_n]$, generated in degree $\leq k$. Set $R=S/I$. Let $u_1, \ldots, u_N$
be a list of all monomials in $S$ of degree $d$, with $N=\binom{n+d-1}{d}$. Consider the \emph{$d$-th Veronese rings}
\[S^{(d)}=\bigoplus_{i\geq 0}S_{di}\subset S \mbox{ \ \ \ and \ \ \ } R^{(d)}=\bigoplus_{i\geq 0}R_{di}\subset R.\]
If $T$ is the polynomial ring $\KK[y_1, \ldots, y_N]$, we have natural surjections
\[\xymatrix{
T \ar@{->>}[r]^{\phi_d} & S^{(d)} \ar@{->>}[r]^{\psi_d} & R^{(d)}.
}\]
(Here $\phi_d$ is the map induced by $y_i \mapsto u_i$, and $\psi_d$ is the restriction to $S^{(d)}$ of the projection from $S$ to $S/I$.) If we set $\pi_d=\psi_d\circ\phi_d$, we can define
\[V_d(I)=\Ker \pi_d=\Ker \phi_d + \phi_d^{-1}(I\cap S^{(d)}).\]
Since $\Ker \phi_d$ is generated by quadrics, $V_d(I)$ is generated in degree $\le \max\{2,\lceil k/d \rceil\}$.
Furthermore, we have that the graphs $G(I)$ and $G(V_d(I))$ are the same, since
\[\Proj(R)\cong \Proj(R^{(d)})\]
as projective schemes. Finally, since $R^{(d)}$ is a direct summand of $R$ and $R$ is integral over $R^{(d)}$, then $R^{(d)}$ is Cohen-Macaulay whenever $R$ is, by a theorem of Eagon and Hochster \cite[Theorem 6.4.5]{BrunsHerzog}. 

This allows us to show how Conjecture  \ref{con:main} has implications for the diameter of the dual graphs of all ideals, not only of those generated in degree 2. 

\begin{proposition} \label{prop:RedToQuadrics}
Suppose Conjecture \ref{con:main} is true. Let $I \subset S$ be an ideal generated in degree $\le k$. If $S/I$ is Cohen--Macaulay, then \[\diam G(I) \leq \frac{(n+\lfloor (k-1)/4\rfloor)^{\lceil k/2\rceil}}{\lceil k/2\rceil !}.\]
\end{proposition}

\begin{proof}
With the notation above, set  $d=\lceil k/2\rceil$ and $e=\lfloor (k-1)/4\rfloor$. Then $V_d(I)$ is quadratic and $G(I)=G(V_d(I))$. Furthermore, $T/V_d(I)$ is Cohen--Macaulay, because $S/I$ is. Assuming Conjecture \ref{con:main}, we get 
\[\diam G(I)= \diam G(V_d(I)) \leq N=\binom{n+d-1}{d} = \frac{(n+d-1) \cdots n}{d!} \le \frac{(n+e)^d}{d!}. \qedhere\]
\end{proof}

\subsection{Reduction to projective curves}
Here we show that under some extra technical assumption (satisfied by subspace arrangements, for example) Conjecture \ref{con:main} can be further reduced to the case where $I$ defines a projective curve. The geometric intuition is to intersect our algebraic object in $\PP^{n}$ with a hyperplane in general position, so that the intersection, viewed as an algebraic object in $\PP^{n-1}$, has the same dual graph as the starting object. 

Throughout this section, we require $\KK$ to be an infinite field (not necessarily algebraically closed).

\begin{lemma}\label{lem:redtodim2}
Let $I\subset S$ be a radical homogeneous ideal such that $S/I$ is a $d$-dimensional Cohen--Macaulay ring, with $d\geq 3$. If $S/\pp$ is Cohen--Macaulay for all $\pp\in\Min(I)$, then there exists a  radical homogeneous ideal $I'\subset S'=\KK[x_1,\ldots ,x_{n-1}]$ such that 
\begin{compactenum}[\rm (i)]
\item $S'/I'$ has dimension $d-1$, 
\item for each $i$ and $j$ the graded Betti number $\beta_{i,j}(S'/I')$ equals $\beta_{i,j}(S/I)$, and 
\item $G(I')=G(I)$. 
\end{compactenum} 
Furthermore, there is a bijection $\phi:\Min(I)\rightarrow \Min(I')$ such that for each $i$ and $j$, $\beta_{i,j}(S'/\phi(\pp))=\beta_{i,j}(S/\pp)$ for all $\pp\in\Min(I)$.
\end{lemma}
\begin{proof}
Set $\Min(I)=\{\pp_1,\ldots ,\pp_s\}$, $R=S/I$ and $R_i=S/\pp_i$. By making a change of coordinates we can assume that $x_n\in S$ is general, so we have that $A=R/(x_n)$ and $A_i=R_i/(x_n)$ are $(d-1)$-dimensional Cohen--Macaulay rings. Since $d-1\geq 2$, we have $H_{\mm}^0(A_i)=0$, where $H^0_{\mm}$ denotes the $0$-th local cohomology with support in the irrelevant ideal $\mm\subset S$. Therefore, Bertini's theorem tells us that $A_i$ is a domain. This means that 
$\pp'_i= \frac{\pp_i+(x_n)}{(x_n)}$ is a prime ideal contained in $S'=S/(x_n)$. By setting $I'= \frac{I+(x_n)}{(x_n)}\subset S'$ we obtain 
\[\Min(I')=\{\pp'_1,\ldots ,\pp'_s\}.\]
We have that $\height(\pp'_i)=n-d$ and
\begin{align*}
\height(\pp'_i+\pp'_j) =
\begin{cases}
\height(\pp_i+\pp_j) & \text {if } \height(\pp_i+\pp_j)<n,\\
\height(\pp_i+\pp_j)-1 & \text {otherwise}
\end{cases}
\end{align*}
Since $d\geq 3$, we conclude that $G(I')=G(I)$.
\end{proof}

\begin{proposition} \label{prop:RedToCurves}
Let $I\subset S$ be a quadratic ideal defining a subspace arrangement. Assume that $S/I$ is a $d$-dimensional Cohen--Macaulay ring, with $d\geq 3$. Then there exists a quadratic ideal $I'\subset S'=\KK[x_1,\ldots ,x_{n-1}]$, defining another subspace arrangement, such that $S'/I'$ is a $(d-1)$-dimensional Cohen--Macaulay ring and $G(I')=G(I)$.
\end{proposition}
\begin{proof}
A minimal prime $\pp$ of $I$ is generated by linear forms, so clearly $S/\pp$ is Cohen--Macaulay. Lemma \ref{lem:redtodim2} guarantees the existence of the ideal $I'$. To see that $I'$ defines a subspace arrangement, it is enough to prove that $I'$ is radical. This follows immediately from Bertini's theorem and the fact that $S/I$ is Cohen--Macaulay of dimension $>1$.
\end{proof}

The results above allow us to reduce Conjecture \ref{con:main} to the 2-dimensional case.

\begin{corollary}
If Conjecture \ref{con:main} holds when $\dim(S/I)=2$ (that is, when the scheme $\Proj(S/I)$ is a curve), then it also holds for all quadratic ideals $I$ such that, for all $\pp\in\Min(I)$, $S/\pp$ is Cohen--Macaulay. 
\end{corollary}

\begin{corollary}
If Conjecture \ref{con:main} holds when the scheme $\Proj(S/I)$ is a union of lines, then it holds whenever $I$ is quadratic and defines a subspace arrangement.
\end{corollary}

\section{Gorenstein algebras and r-connectivity}\label{sec:gor}
To deal with Gorenstein algebras, we need a tool from liaison theory. Recall that 
inside the polynomial ring $S$, two ideals $I$ and $I'$ without common primary components, are called \emph{geometrically $G$-linked} if
$S/(I \cap I')$ is Gorenstein. (This is stronger than \emph{algebraically $G$-linked}, a property of pairs of ideals widely studied in the literature; cf.\ e.g.\ \cite{Mi}.) Liaison theory easily implies the following result:

\begin{proposition}\label{prop:2-conngor}
Let $I\subset S$ be an ideal such that $S/I$ is Gorenstein. Let $\qq$ be a primary component of $I$. Let $v$ be the vertex of $G(I)$ corresponding to the minimal prime $\pp=\sqrt{\qq}$ of $I$. If $S/\qq$ is Cohen-Macaulay, then 
either 
\begin{compactenum}[\rm (1)]
\item $I$ is primary and $G(I)$ consists only of $v$, or 
\item the deletion of $v$ from $G(I)$ yields a graph $G'$ that is connected.
\end{compactenum}
\end{proposition}

\begin{proof}
Let us write $I=\bigcap_{i=1}^s\qq_i$ where for all $i=1,\ldots ,s$, $\qq_i$ is $\pp_i$-primary. Up to relabeling, we may assume $\qq_1=\qq$. If $s=1$ then $I$ is primary and case (1) is settled, so assume $s\ge 2$. The graph $G(I)$ is on vertices $v=v_1,v_2, \ldots ,v_s$ corresponding to the $\pp_i$'s. Note that $G(I) - v = G(J)$, where 
$J=\qq_2 \cap \qq_3 \cap \ldots \cap \qq_s$.
Now, since $J$ is geometrically linked to $\qq$ by a Gorenstein ideal and $S/\qq$ is Cohen-Macaulay, it follows by the work of Schenzel \cite{Sc} that $S/J$ is Cohen-Macaulay as well (see Migliore \cite[Theorem 5.3.1]{Mi}). In particular $G(J)$ is connected.
\end{proof}

\begin{corollary}\label{cor:2-conngor}
Let $I\subset S$ be an ideal such that $S/I$ is Gorenstein. If $S/\qq$ is Cohen-Macaulay for any primary component $\qq$ of $I$, then either $G(I)$ is a point, or it is a segment, or it is a 2-connected graph. In any case, 
\[\diam G(I) \le \displaystyle \frac{e(S/I)}{2}.\]
\end{corollary}

\begin{proof} $G(I)$ is connected, and by Proposition \ref{prop:2-conngor}, the deletion of any vertex leaves $G(I)$ connected. Let $s$ be the number  of vertices of $G$. By Lemma \ref{lem:menger}, $\diam G(I) \le s/2$, whence we conclude via Lemma \ref{lem:generalbound}. 
\end{proof}

\begin{corollary} \label{cor:Weak}
Let $I\subset S$ be an ideal defining a subspace arrangement. If $S/I$ is Gorenstein, then either $G(I)$ is a point, or it is a segment, or it is a 2-connected graph. 
\end{corollary}

Our goal is now to strengthen the conclusion of Corollary \ref{cor:Weak}. But first, the following examples show that one needs particular caution with the assumptions of Proposition \ref{prop:2-conngor} and its corollaries. First of all, the Cohen--Macaulayness assumption on $S/\qq$ is necessary.

\begin{example} \label{ex:GorNot2conn}
Let $I=(x_0x_3-x_1 x_2,  x_1^2 x_3-x_0 x_2^2) \subset \mathbb{Q}[x_0,\ldots,x_3]=S$. Since $I$ is a complete intersection, $S/I$ is Gorenstein and Cohen--Macaulay.  The prime decomposition of $I$ can be computed with the software \texttt{Macaulay2} \cite{macaulay2}:
\[\sqrt{I} \, = \, I \, = \,  ({x}_{0},{x}_{1})  \: \cap \:  ({x}_{2},{x}_{3})  \: \cap \: 
({x}_{1} {x}_{2}-{x}_{0} {x}_{3},{x}_{2}^{3}-{x}_{1} {x}_{3}^{2},{x}_{0} {x}_{2}^{2}-{x}_{1}^{2} {x}_{3},{x}_{1}^{3}-{x}_{0}^{2} {x}_{2}).\]
The third ideal is the ideal of the projection of a rational normal curve of degree 4: 
\[C=\{[t^4,t^3u,tu^3,u^4]:[t,u]\in\PP^1\}\subset \PP^3.\]
The celebrity of such a quartic curve resides  in the fact that it was studied in Hartshorne's paper \cite{Ha1}, where $C$ was shown to be a set-theoretic complete intersection in positive characteristic. It is unknown whether the same holds in characteristic $0$. However, the coordinate ring of $C$ is \emph{not} Cohen--Macaulay. 
It is easy to see that $G(I)$ is simply a path of two edges, since the primes $({x}_{0},{x}_{1})$ and $({x}_{2},{x}_{3})$ are not connected by an edge. Hence $G(I)$ is $1$-connected, but not $2$-connected. In fact, removing the vertex corresponding to $C$ disconnects the graph.
\end{example}

The ideal of Example \ref{ex:GorNot2conn} is radical. We stress that for non-radical ideals, Proposition \ref{prop:2-conngor} requires the Cohen-Macaulayness of $S/\qq$ (where $\qq$ is the $\pp$-primary ideal), and \emph{not} of $S/\pp$.  The next examples highlight why this distinction is important.

\begin{example} \label{ex:QnotP}
Let $I=(x_4^2-x_3x_5, \;   x_3x_4-x_2x_5, \; x_2x_3-x_1x_5, \;   x_1x_2-x_0x_3) \subset \mathbb{C}[x_0,\ldots,x_5]$. The ideal $I$ is a complete intersection. Its minimal primes are 
\[
\begin{array}{rcl}
\pp_1 &=& \textrm{the prime defining the projective closure of the affine curve } (t, t^3, t^4, t^5, t^6)\\
\pp_2 &=& (x_5,x_4,x_2,x_0),\\
\pp_3 &=& (x_4,x_3,x_2,x_1),\\
\pp_4 &=& (x_5,x_4,x_3,x_2),\\
\pp_5 &=& (x_5,x_4,x_3,x_1).
\end{array}
\]
If $S=\mathbb{C}[x_0,\ldots,x_5]$, clearly  $S/\pp_4$ is Cohen--Macaulay. Using \texttt{Macaulay2} we computed the edges of the graph $G(I)$: they are $13$, $14$, $15$, $24$, $34$, $35$, and $45$. Note that the only vertex adjacent to $2$ is $4$, so deleting $4$ disconnects the graph. How do we reconcile this with Proposition \ref{prop:2-conngor}? If we search for the $\pp_4$-primary ideal in a primary decomposition of $I$, this is not $\pp_4$. It is instead
\[\qq_4=(x_5^2,\; x_4x_5, \; x_4^2-x_3x_5, \; x_3x_4-x_2x_5, \; x_2x_3-x_1x_5, \, x_2^2-x_0x_5, \;x_1x_2-x_0x_3,\; x_3^4)\]
and one can check that $S/\qq_4$ is \emph{not} Cohen--Macaulay.
\end{example}

\begin{example} \label{ex:GorNot2conna}
Let $\pp$ be the prime homogeneous ideal in $S=\Z_2[x_1,\ldots,x_6]$ defining the projective curve
\[
\begin{array}{c}
(t^5u + t^4u^2 + u^6, \; t^5u + t^4u^2 + t^2u^4 + u^6,  \; t^5u + t^4u^2 + tu^5, \; t^6 + t^3u^3 + tu^5, \\ 
t^6 + t^5u + t^4u^2 + t^3u^3 + t^2u^4 + u^6, \: t^6 + t^5u + t^3u^3 + tu^5 + u^6)
\end{array}
  \subset \PP^5.\]
One can see with \texttt{Macaulay2} that $S/\pp$ is not Cohen-Macaulay and $\pp$ is generated by the $8$ quadratic polynomials
\[\!\!
\begin{array}{l}
a = x_4^2 + x_1x_5 + x_4x_5 + x_4x_6 + x_5x_6,\\
b = x_2x_3 + x_3x_4 + x_1x_5 + x_3x_6 + x_4x_6 + x_5x_6 + x_6^2,\\
c = x_2x_4 + x_3x_4 + x_1x_5 + x_3x_5 + x_5^2 + x_4x_6 + x_5x_6,\\
d = x_2^2 + x_1x_4 + x_3x_4 + x_2x_5 + x_3x_5 + x_4x_5 + x_5^2 + x_1x_6 + x_2x_6 + x_3x_6 + x_4x_6 + x_5x_6 + x_6^2,\\
e =  x_3^2 + x_3x_5 + x_5^2 + x_1x_6 + x_4x_6,\\
f = x_1x_3 + x_1x_4 + x_1x_5 + x_2x_5 + x_5^2 + x_1x_6 + x_2x_6 + x_3x_6 + x_4x_6 + x_5x_6,\\
g = x_1x_2 + x_3x_4 + x_2x_5 + x_3x_5 + x_4x_5 + x_1x_6 + x_3x_6 +x_5x_6, \\
h = x_1^2 + x_1x_5 + x_4x_5 + x_5^2 + x_2x_6 + x_4x_6 + x_6^2.
\end{array}
\]
The ideal $I_1=(a,c,f,g)$ is a complete intersection and has radical equal to $\pp$, so $\pp$ is a set-theoretic complete intersection. $G(I_1)$ consists of a single point. 

The ideal $I_2=(b,f,g,h)$ is a complete intersection whose radical is strictly contained in $\pp$. The minimal primes of $I_2$ are
\[
\begin{array}{rcl}
\pp_1 &=& \pp\\
\pp_2 &=& ({x}_{6},{x}_{4}+{x}_{5},{x}_{2}+{x}_{5},{x}_{1}),\\
\pp_3 &=& ({x}_{6},{x}_{5},{x}_{3},{x}_{1}),\\
\pp_4 &=& ({x}_{5}+{x}_{6},{x}_{3}+{x}_{6},{x}_{2},{x}_{1}+{x}_{6}),\\
\pp_5 &=& ({x}_{4}+{x}_{5}+{x}_{6},{x}_{3}+{x}_{5},{x}_{2}+{x}_{5}+{x}_{6},{x}_{1}+{x}_{6})
\end{array}
\]
Hence the graph $G(I_2)$ consists of the edges $12,14,15, 25,34,45$. In particular, $G(I_2)$ has diameter $3$. Since $3$ is a leaf (only $4$ is adjacent to it), $G(I_2)$ is not $2$-connected. As in Example \ref{ex:QnotP}, $S/\pp_4$ is Cohen--Macaulay, but $S/\qq_4$ is not, where $\qq_4$ is the $\pp_4$-primary component.

Finally, the ideal  $I_3=(c,f,g,h)$ is again a complete intersection with radical strictly contained in $\pp$. The minimal primes of $I_3$ are
\[
\begin{array}{rcl}
\pp'_1 &=& \pp\\
\pp'_2 &=& (x_6,x_4+x_5,x_2+x_5,x_1),\\
\pp'_3 &=& (x_6,x_5,x_4,x_1),\\
\pp'_4 &=& (x_5+x_6,x_3+x_6,x_2,x_1+x_6),\\
\pp'_5 &=& (x_3+x_4,x_2+x_4+x_6,x_1+x_4+x_5+x_6,x_4^2+x_5^2+x_5x_6+x_6^2),\\
\pp'_6 &=& (x_4+x_5+x_6,x_3+x_5,x_2+x_5+x_6,x_1+x_6)
\end{array}
\]
The graph $G(I_3)$ has edges 12, 14, 15, 16, 23, 25, 26, 45, 46, 56. Such a graph has diameter 3 and is not $2$-connected: The vertex 3 is adjacent only to 2. As above, $S/\pp_2$ is Cohen--Macaulay but $S/\qq_2$ is not, where $\qq_2$ is the $\pp_2$-primary component.
\end{example}


Next we show that the conclusion ``$2$-connected'' of Proposition \ref{prop:2-conngor} is best possible.

\begin{example} \label{ex:GorNot3conn}
Let $J$ be the homogeneous ideal of $S=\mathbb{Q}[x_0, ..., x_4]$ given by 
\[J=(-x_1x_2+x_0x_3, \; -x_2^2+x_1x_3,  \;-x_1x_3+x_0x_4).\]
$J$ is a complete intersection, hence in particular $S/J$ is Gorenstein (of Castelnuovo--Mumford regularity $3$). One of the minimal primes $\pp_1$  of $J$ is well known, as it defines the rational normal curve 
\[C=\{[t^4,t^3u,t^2u^2, tu^3,u^4]:[t,u]\in\PP^1\}\subset \PP^4.\]
The other primes are $\pp_2=(x_0,x_1,x_2)$,
$\pp_3=(x_0,x_2,x_3)$, 
$\pp_4=(x_2,x_3,x_4)$.
$J$ is ``almost'' radical: a primary decomposition of $J$ is 
\[J= \pp_1 \cap \qq_2 \cap \pp_3 \cap \pp_4,\]
where $\qq_2=(x_0,x_1,x_2^2)$ is $\pp_2$-primary. For each primary component $\qq$ of $J$, $S/\qq$ is a Cohen-Macaulay (and even level) algebra. However, $G(J)$ is \emph{not} the complete graph on $4$ vertices, because the edge between $\pp_2$ and $\pp_4$ is missing. (All other edges are there, so $G(J)$ is $K_4$ minus an edge.) In particular, $G(J)$ is $2$-connected, but not $3$-connected: The deletion of the vertices corresponding to $\pp_1$ and $\pp_3$ disconnects it. 
\end{example}

With all these careful distinctions in mind, we are ready to announce our main result.

\begin{theorem}\label{thm:GorensteinRConnected}
Let $I\subset S$ be the defining ideal of a subspace arrangement. If $S/I$ is Gorenstein of Castelnuovo--Mumford regularity $r$, then $G(I)$ is $r$-connected.
\end{theorem}

\begin{proof}
Let $\overline{\KK}$ be the algebraic closure of $\KK$, $S'=S\otimes_{\KK}\overline{\KK}$ and $I'=IS'$. Since $S\hookrightarrow S'$ is faithfully flat, we have that $S'/I'$ is Gorenstein and has regularity $r$. Furthermore, if $I=\pp_1 \cap \ldots \cap \pp_s$, again, by the flatness we have $I'=\pp_1S' \cap \ldots \cap \pp_s S'$. Extensions of prime ideals are not prime in general, but since our $\pp_i$'s are generated by linear forms, the $\pp_iS'$ are also prime ideals. So, $I'$ is the defining ideal of a subspace arrangement, and $G(I')=G(I)$. This means there is no loss in assuming that $\KK$ is algebraically closed. 

Let $d=\dim(S/I)$. By Lemma \ref{lem:redtodim2}, we can assume that $d=2$. This has the advantage that ``connected in codimension one'' is the same as ``connected''. Let us write
\[I=\bigcap_{i=1}^s\pp_i \subset S=\KK[x_1,\ldots ,x_n]\]
where the $\pp_i$'s are ideals generated by linear forms and have height $n-2$. For the rest of the proof, for any subset $A\subset \{1,\ldots ,s\}$ we set $I_A=\bigcap_{i\in A}\pp_i$.

To show that $G(I)$ is $r$-connected, we must verify that $G(I_A)$ is connected for any subset $A\subset \{1,\ldots ,s\}$ such that $|\{1,\ldots ,s\}\setminus A|<r$. Notice that, because $I_A$ is radical and $\KK$ is algebraically closed, we have:
\[G(I_A) \mbox{ is connected } \iff C_A \mbox{ is connected } \iff H^0(C_A,\O_{C_A})\cong \KK \iff H_{\mm}^1(S/I_A)_0=0,\]
where $C_A$ is the curve $\Proj(S/I_A)\subset \PP^{n-1}$ and $\mm$ is the irrelevant ideal of $S$. 

Set $B=\{1,\ldots ,s\}\setminus A$, $I_B=\cap_{i\in B}\pp_i$ and $C_B=\Proj(S/I_B)$. Then $C_A$ and $C_B$ are geometrically linked by $C=\Proj(S/I)$, which is arithmetically Gorenstein. By Schenzel's work~\cite{Sc} (see also \cite[Theorem 5.3.1]{Mi}) we have a graded isomorphism
\[H_{\mm}^1(S/I_A)\cong H_{\mm}^1(S/I_B)^{\vee}(2-r), \]
where $-^{\vee}$ means $\Hom_{\KK}(-,\KK)$. Therefore $H_{\mm}^1(S/I_A)_0$ is nonzero if and only if there is a nonzero map of $\KK$-vector spaces from $H_{\mm}^1(S/I_B)$ to $\KK$ of degree $2-r$, if and only if $H_{\mm}^1(S/I_B)_{r-2}\neq 0$. However, by the main result of Derksen and Sidman \cite{DS},
\[\reg(S/I_B)=\reg(I_B)-1\leq |B|-1< r-1,\]
so that $H_{\mm}^1(S/I_B)_j=0$ for all $j\geq r-2$ by Equation \eqref{eq:grothendieck}, and this concludes the proof.
\end{proof}

\begin{remark}\label{rem:RegBestPossible}
It is natural to ask whether Theorem \ref{thm:GorensteinRConnected} can be extended from the generality of subspace arrangements, to arbitrary ideals. The answer is negative. In fact, Example \ref{ex:GorNot3conn} presents an ideal $J$ such that $S/J$ is Gorenstein and has Castelnuovo-Mumford regularity $3$, yet $G(J)$ is not $3$-connected. Another example would be given by the complete intersection 
$I=(x_4^2-x_3x_5,  \;  x_1x_4-x_0x_5,   \;   x_2x_3-x_1x_5,  \;  x_1x_2-x_0x_3)$: the graph $G(I)$ is $2$- but not $3$-connected, while $\reg(S/I)=4$.

Similarly, one could ask whether Theorem \ref{thm:GorensteinRConnected} can be extended from Gorenstein to Cohen--Macaulay subspace arrangements. The answer is once again negative, already for coordinate subspace arrangements. For example, let $\Delta$ be the graph $12$, $13$, $23$, $14$, $45$. The Stanley--Reisner ring $\KK[x_1, \ldots, x_5]/I_\Delta$ is Cohen--Macaulay of regularity $2$. However, $G(I_\Delta)$ is connected, but not $2$-connected. 
\end{remark}

\begin{corollary}[Klee \cite{Klee}]  \label{cor:klee}
Let $I=I_\Delta$ be the Stanley-Reisner ideal of a homology $d$-sphere $\Delta$. Then $G(I)$ is $(d+1)$-connected.
\end{corollary}

\begin{proof}
By Hochster's formula \cite[Corollary 5.12]{MiSt}, if $\Delta$ is a homology $d$-sphere, then its Stanley--Reisner ring is Gorenstein of regularity $d+1$.
\end{proof} 

\begin{corollary}[Balinski]  \label{cor:balinski}
If $P$ is any (simple) $d$-dimensional convex polytope, the $1$-skeleton of $P$ is $d$-connected.
\end{corollary}

\begin{proof}
The $1$-skeleton of $P$ is the dual graph of the simplicial $d$-sphere $\Delta=\partial P^*$, where $P^*$ is the polytope polar dual to $P$. By Corollary \ref{cor:klee}, we conclude.
\end{proof}

\begin{corollary} \label{cor:CIdc-cConnected}
Let $I$ be a complete intersection of height $c$ defining a subspace arrangement, and let $d$ be the minimal degree of a generator of $I$. Then $G(I)$ is $(d-1)c$-connected.
\end{corollary}

\begin{proof}
If $I=(f_1,...,f_c)$, then the Castelnuovo-Mumford regularity of $S/I$ is $\deg(f_1) + ... +\deg(f_c) - c\geq (d-1)c$.
\end{proof}

It is easy to see that the connectivity bounds given by Theorem \ref{thm:GorensteinRConnected} and Corollaries \ref{cor:klee} and \ref{cor:CIdc-cConnected}, cannot be improved in general:

\begin{example} \label{ex:RegBestPossible}
Let $I_r = (x_1x_2, x_3x_4, \ldots, x_{2r-1}x_{2r}) \subset K[x_1, \ldots, x_{2r}] =S$. This $I_r$ is the Stanley-Reisner ring of the boundary of the $r$-dimensional crosspolytope. Since $\height(I_r)=r$, the ideal $I_r$ is a complete intersection. Moreover, the regularity of $S/I_r$ is exactly $r$. By Lemma \ref{lem:combcomm}, $G(I_r)$ is the dual graph of the $r$-crosspolytope, or in other words, the $1$-skeleton of the $r$-cube. So $G(I_r)$ is $r$-connected.  However, every vertex of $G(I_r)$ has degree $r$, so $G(I_r)$ is not $(r+1)$-connected.
\end{example}

\section{Arrangements of lines canonically embedded} \label{sec:arrangement}
Let $C$ be an arrangement of projective lines. Consider the graph $G(C)$ whose vertices correspond to the irreducible components of $C$, and such that two vertices are connected by an edge if and only if the intersection of the two corresponding irreducible components is nonempty. Once $C$ is embedded in some $\PP^N$, we have $G(C)=G(I)$, where $I$ is the ideal defining $C$. Whether this defining ideal $I$ is quadratic or not depends on the embedding; and the same is true for whether $S/I$ is Cohen-Macaulay.  
In this section, we will prove bounds on $\diam G(I)$ for a certain, special embedding of $C$, called ``canonical embedding''. Such an embedding does not always exist, but when it does, it tends to produce defining ideals that are both quadratic and Cohen--Macaulay. 

\begin{remark}\label{rem:notallgraphsaredual}
There are graphs $G$ which cannot be realized as dual graphs of arrangements of projective lines. For example, take the graph 
\[ G = \{12, 13, 14, 15, 16, 23, 24, 26, 35, 36, 45, 46, 56\},\] 
(which is $K_6$ minus two disjoint edges.) An arrangement $C$ of projective lines such that $G(C)=G$ would consist of $6$ projective lines $r_1$, $r_2$, $r_3$, $r_4$, $r_5$ and $r_6$. Let $P$ be the point $r_1 \cap r_2$. The three lines $r_1$, $r_2$ and $r_3$ are pairwise incident. So there are two cases: either $r_3$ passes through $P$, or not.

In the first case, $r_1$, $r_2$ and $r_3$ are not coplanar, because $r_4$ touches two of them but not all of them. So $r_4$ touches $r_1$  and $r_2$ in two points $Q$ and $R$, respectively, which are both different than $P$. Hence $r_1$, $r_2$ and $r_4$ all belong to the same plane. The fifth line $r_5$ cannot belong to such plane, because it does not intersect $r_2$. But $r_5$ meets both $r_1$ and $r_4$. So $r_5$ passes through the point $Q=r_1 \cap r_4$ and intersects $r_3$ in a further point $S$. The five lines of the arrangement are then contained in the union of two planes, determined by the two triangles $PQR$ and $PQS$; at the intersection of the two planes lies the line $r_1$. So we reached a contradiction, because there cannot be a sixth line $r_6$, different than $r_1$, yet incident to all lines of the arrangement.

In the second case, i.e. if $r_3$ does not pass through $P$, the three lines $r_1$, $r_2$ and  $r_3$ must belong to a common plane. The lines $r_4$ and $r_5$ cannot belong to such plane, since $r_4$ does not intersect $r_3$, and $r_5$ does not intersect $r_2$. Hence $r_4$ passes through $P=r_1 \cap r_2$, and $r_5$ passes through $Q':=r_1 \cap r_3$.
Set $R':=r_2 \cap r_3$ and $S':=r_4 \cap r_5$. As before, this five-line arrangements determines two planes, intersecting at the line $r_1$; so there cannot be a sixth line incident to all five lines.

As a consequence, any graph containing the above $G$ as an induced subgraph cannot be realized as the dual graph of arrangements of projective lines.
\end{remark}

Analogously to the proof of Lemma \ref{lem:redtodim2}, one can show that, for any pure simplicial complex $\Delta$, an arrangement of projective lines $C$ obtained by taking general hyperplane sections of the coordinate arrangement defined by $I_{\Delta}$ satisfies $G(C)=G(\Delta)$. Therefore Remark \ref{rem:notallgraphsaredual} implies the following:

\begin{corollary} \label{cor:notallgraphsaredual}
Some graph is not the dual graph of any pure simplicial complex.
\end{corollary}

We now need some algebraic geometry notation; we refer the reader to the standard  textbook by Hartshorne \cite[Chapter II.7]{Hartshorne} for proofs and further details. 

Given an invertible sheaf $\mathcal{L}$ on $C$, if $C$ is a projective curve the $\KK$-vector space $\mathcal{L}(C)$ is finite. Let us consider a basis $s_0, \ldots, s_N$ of $\mathcal{L}(C)$. The elements of  $\mathcal{L}(C)$ are called \emph{global sections}. By \cite[Chapter II, Theorem 7.1]{Hartshorne}, there is a unique morphism $\phi: C \rightarrow \PP^N$ such that $\mathcal{L}$ is isomorphic to the pull-back $\phi^* (\mathcal{O}_{\PP^N}(1))$ and $s_i = \phi^* (x_i)$, where the $x_i$'s are the coordinate functions on $\PP^N$. 
In particular, $\mathcal{L}(C)$ is isomorphic as vector space to $S_1$, where $S=\KK[x_0, \ldots, x_N]$.
The sheaf $\mathcal{L}$ is called \emph{very ample} if this morphism $\phi$ is an immersion. 

If $P$ is an arbitrary point on the curve $C$, we denote by $\mathcal{L}_P$ the stalk of  $\mathcal{L}$ at $P$. By $\mm_P$ we denote the maximal ideal of the local ring $\mathcal{O}_{C,P}$. For any global section $s$ in $\mathcal{L}(C)$, $s_P$ will denote the image of $s$ in the stalk $\mathcal{L}_P$. The zero locus of $s$ is
\[(s)_0=\{P \textrm{ in } C  \textrm{ such that } s_P \in \mm_P \mathcal{L}_P\}.\] 
With the notation above, one can prove the following well-known fact:

\begin{lemma} \label{lem:VeryAmple}
If $\mathcal{L}$ is very ample, $s$ is a global section of $\mathcal{L}$ and $\ell$ is the unique element of  $S_1$ such that $\phi^* (\ell) = s$, then the points of $(s)_0$ correspond to the points of intersection between the curve $C$ and the hyperplane defined by $\ell$. 
\end{lemma}

A curve $C$ is called \emph{locally Gorenstein} if all the stalks $\mathcal{O}_{C,P}$, where $P$ ranges over the points of $C$, are Gorenstein rings.  

\begin{lemma} \label{lem:LG}
Any arrangement of projective lines is locally Gorenstein, provided no three lines of the arrangement meet in a common point.
\end{lemma}

\begin{proof}
If $P$ belongs to one line only, then $\mathcal{O}_{C,P}$ is even a regular ring. Otherwise $\mathcal{O}_{C,P}$ has Krull dimension $1$ and  embedding dimension $2$. In particular, it is Gorenstein. 
  \end{proof}

On a locally Gorenstein curve $C$, one can define another invertible sheaf, called \emph{canonical sheaf} and usually denoted by 
 $\omega_C$. (It coincides with the dualizing sheaf defined in \cite[Chapter III, Section 7]{Hartshorne} for any projective scheme $X$. By definition of Gorenstein ring, the dualizing sheaf is invertible if and only if the scheme is locally Gorenstein.) The \emph{genus} of the curve $C$ is the dimension of the finite vector space $\omega_C (C)$. The genus has a particularly nice interpretation if $C$ is an arrangement of projective lines. 
 
 \begin{proposition}[{Bayer--Eisenbud \cite[Proposition 1.1]{BE}}] \label{prop:BayerEisenbud}
Let $C$ be an arrangement of projective lines. If no three lines of $C$ meet at a common point, then the genus of $C$ equals $t - s + 1$, where $t$ (resp. $s$) is the number of edges (resp. vertices) of $G=G(C)$. 
 \end{proposition}

 When the canonical sheaf is very ample, it defines (as we saw for $\mathcal{L}$) an immersion $\phi': C \hookrightarrow \PP^N$, which is usually called the \emph{canonical embedding}. With slight abuse of notation, we use the expression ``$C$ canonically embedded'' to denote the image $\phi'(C) \subset \PP^N$. It is well known that canonical embeddings play a central role in the theory of nonsingular curves: If the genus of the curve is at least 3, typically $\omega_C$ is very ample and the corresponding ideal is quadratic and Cohen-Macaulay (compare \cite[Chapter 9]{Ei}). For the purposes of the present paper this is not interesting, since (connected) nonsingular curves are irreducible. However, a similar philosophy holds also for reducible curves (see \cite{BE}).

\begin{lemma}\label{lem:ampleness}
Let $C$ be an arrangement of projective lines, in which no three lines  meet at a common point. 
If the canonical sheaf $\omega_C$ is very ample, then $G(C)$ is 3-edge-connected.
\end{lemma}
\begin{proof}
First of all, the existence of $\omega_C$ is guaranteed by Lemma \ref{lem:LG}. 
By contradiction, we can find two distinct edges in the graph $G(C)$ whose removal disconnects it. Let $P, Q$ be the two points on the curve $C$ corresponding to these two edges. 
Let us consider the subspace of $\omega_C(C)$
\[W= \{ 
s \in \omega_C(C) \textrm{ such that } (s)_0 \textrm{ contains both } P \textrm{ and } Q
\}.\]
\noindent By \cite[Proposition 2.3]{BE}, $W$ has codimension $1$ in $\omega_C(C)$.  Now we use the assumption that $\omega_C$ is very ample, or in other words, that the morphism $\phi': C \hookrightarrow \PP^{N}$ is an immersion. Let $V$ be the $\KK$-vector space formed by the linear forms of $S=\KK[x_0, \ldots, x_N]$ that vanish on both $P$ and $Q$. By Lemma \ref{lem:VeryAmple},
$W$ is isomorphic as vector space  to $V$. However,  $V$ has codimension $2$ in $S_1$. But $S_1$  is isomorphic to $\omega_C(C)$, in which $W$ has codimension~$1$: A contradiction. 
\end{proof}

Recall that a height-unmixed ideal $I$ is Hirsch if the diameter of $G(I)$ is $\le \height(I)$.

\begin{theorem}\label{thm:LinesArrangement}
Let $C\subset \PP^{N}$ be an arrangement of lines no three of which meet at a common point. If $C$ is canonically embedded, then its defining ideal $I$ is Hirsch.
\end{theorem}

\begin{proof}
First of all, notice that $N=g-1$ where $g$ is the genus of $C$.
Let $s$ (resp. $t$) be the number of vertices (resp. edges) of the graph $G(C)$. The ideal $I$ has height $g-2$, where $g$ is the genus of the curve. By Proposition \ref{prop:BayerEisenbud}, $g = t - s + 1$, and by Lemma \ref{lem:ampleness} $G$ is $3$-edge-connected. In particular, every vertex of $G$ lies in at least $3$ edges and $s\geq 4$, which implies $2 t \ge 3 s$. If $s < 2t/3$, then
\[\height I = g-2=t-s-1 >  t/3 -1,\]
which, since $\height I$ is an integer, implies $\height I\geq \lfloor t/3\rfloor$. Now Lemma \ref{lem:menger} (b) implies $\diam G\leq \height I$. 

If $2t=3s$, then $G$ is trivalent, that is: Each vertex of $G$ lies in exactly 3 edges. A 3-edge connected trivalent graph is also 3-connected by \cite[Lemma 2.6]{BE}, so Lemma \ref{lem:menger} (a) and the fact that $s\geq 4$ let us conclude because:
\[\height I = g-2=t-s-1 =  s/2 -1 = (s-2)/2\geq \lfloor (s-2)/3\rfloor - 1. \qedhere\]

%
\end{proof}

\section{Further examples of Hirsch and non-Hirsch ideals} \label{sec:examples}

In this section we prove the Hirsch property for a few cases, including all  ideals of small height or regularity.

\begin{proposition}\label{prop:easy}
The following homogeneous ideals of $S=\KK[x_1, \ldots, x_n]$ are Hirsch:
\begin{compactenum}[\rm (i)]
\item prime ideals;
\item ideals corresponding to finite sets of points;
\item ideals of height $1$ (that is, hypersurfaces);
\item ideals such that $S/I$ is Cohen--Macaulay of regularity $1$;
\item height-unmixed ideals in a polynomial ring with $n\le 3$ variables.
\end{compactenum}
\end{proposition}

\begin{proof}
\begin{compactenum}[(i)]
\item $G(I)$ is a single point. 
\item In this case, $G(I)$ is the complete graph on $s$ vertices. So $\diam G(I) = 1 \le \height I$.
\item For any two primes $\pp_i$, $\pp_j$ of $S$, one has $\height (\pp_i + \pp_j) \le \height \pp_i + \height \pp_j$. So if $\height(I)=1$, for any two different minimal primes $\pp_i, \pp_j$ of $I$ we have $\height (\pp_i) = \height(\pp_j) = 1$ and $\height (\pp_i + \pp_j) = 2$. So $G(I)$ is the complete graph, as above. 
\item Being $G(I)$ connected, $\diam G(I) \le s - 1$, where $s$ is the number of vertices of $G(I)$; but by Lemma \ref{lem:regularity}, part (ii), we have $s \le \height(I) + 1$. 
\item Let $I\subset \KK[x_1,x_2,x_3]$. If the height of  $I$ is 1 resp. 2 resp. 3,  we conclude by part (iii) resp. (ii) resp. (i). {$\qedhere$}
\end{compactenum}
\end{proof}

However, it is easy to find non-Hirsch ideals in a polynomial ring with four or more variables:
\begin{example} \label{ex:nonHirsch}
The dual graph of the ideal 
\[I=(x_1,x_2)\cap (x_2,x_3)\cap (x_3,x_4)\cap (x_4,x_1+x_3)\]
is a path of three edges, hence has diameter $3$. Since $\height(I)=2$, $I$ is \emph{not} Hirsch. 
\end{example}

Note that $x_1x_3x_4$ is a minimal degree-3 generator for $I$, so $I$ is not generated by quadrics. 
In fact, height-$2$ (unmixed) ideals generated by quadrics are all Hirsch:

\begin{proposition}
Let $I\subset S$ be a height-unmixed ideal of height $c\geq 2$. If all the minimal generators of $I$ have degree $\leq d$ and $G(I)$ is connected, then $\diam G(I)\leq d^c-2$. 
\end{proposition}
\begin{proof}
If $d=1$, this is obvious, so we can assume $d\geq 2$. Notice that, since $G(I)$ is connected, $I$ is height-unmixed. Therefore the number of vertices of $G=G(I)$ is mostly $d^c$ by Lemma \ref{lem:regularity}. So the only case in which the bound in the statement could fail is if $G$ was a path on $d^c$ vertices. In such a case, however, $I$ would be a complete intersection of degree-$k$ polynomials defining a subspace arrangement, so $G$ would be $c$-connected by Corollary \ref{cor:CIdc-cConnected}. We thus conclude by Lemma \ref{lem:menger}.
\end{proof}

\begin{corollary}
\label{cor:height2}
Let $I$ be a height-$2$ ideal, generated by quadrics. If  $S/I$ is Cohen--Macaulay, then $I$ is Hirsch.
\end{corollary}

\begin{proposition} \label{prop:reg2}
If $S/I$ is Gorenstein of regularity $2$, then $I$ is Hirsch.
\end{proposition}

\begin{proof}
If $I$ contains linear forms, we can quotient them out without changing the regularity, so there is no loss in assuming $I\subset \mm^2$.

Since Gorenstein implies Cohen--Macaulay, by Lemma \ref{lem:deghreg} the $h$-polynomial of $S/I$ has degree $2$. Moreover, recall that if $S/I$ is Gorenstein, then the $h$-polynomial is palyndromic. Set $c=\height(I)$; we have 
\[e(S/I)=h(1)=h_0 + h_1 + h_2=2 + h_1 = 2+ c.\]
We distinguish two cases: either the number of vertices of $G(I)$ is $s \le e(S/I) - 1$, or $s = e(S/I)$. If $s \le e(S/I) - 1$, from the connectedness of $G(I)$ we have 
\[\diam G(I) \; \le \; s-1 \; \le \; e(S/I) - 2 \; = \; \height I.\] 
So, the only case left is when $s=e(S/I)$, that is, when $I$ defines a subspace arrangement. In this case, by Corollary \ref{cor:2-conngor} and Lemma \ref{lem:menger} we obtain
\[\diam G(I) \; \le \; \left \lfloor \frac{s-2}{2} \right\rfloor + 1 \; \leq \; \left\lfloor \frac{c}{2} \right\rfloor +1 \; \le c. \qedhere\]
\end{proof}

In Proposition \ref{prop:reg2}, note that $I$ is quadratic unless it defines a hypersurface.

\subsection{An ideal with many quadratic minimal primes}
The intuition seems to suggests that, in dealing with Conjecture \ref{con:main}, the hardest case should be when $I$ defines a subspace arrangement, because this is the case where $G(I)$ has more vertices (cf. Remark \ref{rem:vgsa}). For this reason in the present paper we focused mostly on this case. However one can also find examples of quadratic complete intersections $I$ such that $\Min(I)$ consists of many quadratic prime ideals. We study the graph $G(I)$ in one such example, pointed out to us by Aldo Conca and Thomas Kahle, and prove it is anyway Hirsch.

\begin{example}
Let $X=(x_{ij})$ be an $m\times m$- symmetric matrix ($x_{ij}=x_{ji}$) of indeterminates over $\KK$, $S=\KK[X]$ the corresponding polynomial ring in $\binom{m+1}{2}$ variables and $I$ the ideal generated by the principal 2-minors of $X$, namely
\[I=(x_{ii}x_{jj}-x_{ij}^2:1\leq i<j\leq m).\]
The ideal $I$ is a complete intersection of quadrics of height $\binom{m}{2}$. Below, we are going to show that the graph $G(I)$ has $2^{\binom{m-1}{2}}$ vertices, and we will describe the corresponding minimal prime ideals of $I$.

Notice that $I$ is contained in the ideal $I_2(X)$ generated by all the 2-minors of $X$, which is a prime ideal of the same height $\binom{m}{2}$. Therefore $I_2(X)\in\Min(I)$. We can find many other minimal primes like this: If $g$ is a change of variables of $S$, we denote
\[gX=(g(x_{ij}))\]
Evidently the ideals $I_2(gX)\subset S $ have the same properties of $I_2(X)$: They are prime ideals of height $\binom{m}{2}$, $S/I_2(gX)$ is a Cohen--Macaulay ring of multiplicity $2^{m-1}$, and so on. Now, let $G$ be the set of changes of variables that
fix the variables $x_{ii}$ and change sign to some $x_{ij}$'s with $i<j$. For any $g \in G$, we have $I\subset I_2(g(X))$. Hence
\[\{I_2(gX):g\in G\}\subset \Min(I).\] 
We want to show that equality holds. Since the multiplicity of $S/I$ is $2^{\binom{m}{2}}$, by the additivity of the multiplicity it is enough to show that
\[|\{I_2(gX):g\in G\}|=2^{\binom{m-1}{2}}.\]
Certainly $|\{I_2(gX):g\in G\}|\leq 2^{\binom{m-1}{2}}$, since $|\{I_2(gX):g\in G\}|e(S/I_2(X))\leq e(S/I)$ by \eqref{eq:addrad}. So we must produce $2^{\binom{m-1}{2}}$ elements $g\in G$ such that the ideals $I_2(gX)$ are pairwise different (notice that $|G|=\binom{m}{2}$). To this end, for any subset $A\subset \{(i,j):1\leq i<j\leq m\}$ let us denote by $g_A$ the change of variables given by
\begin{align*}
g_A(x_{ij}) =
\begin{cases}
x_{ij} & \text {if } (i,j)\notin A,\\
-x_{ij} & \text {if } (i,j)\in A.
\end{cases}
\end{align*}
Now let us fix $U=\{(i,j):1\leq i<j-1\leq m-1\}$. The set $U$ has cardinality $\binom{m-1}{2}$ and, if $A$ and $B$ are different subsets of $U$, one has $I_2(g_AX)\neq I_2(g_BX)$. To see this, we can assume that there is a $j$ such that for some $i$, $(i,j)\in A\setminus B$. Pick the maximum index $i$ doing the job, and notice that $i\leq m-2$ (since $A$ is in $U$). By denoting $\lbrack a,b \mid c,d\rbrack_{gX}$ the 2-minor of $gX$ corresponding to the rows $a,b$ and the columns $c,d$, we have:
\begin{eqnarray*}
\lbrack i, i+1\mid i+1,j\rbrack_{g_AX}=\delta x_{i,i+1}x_{i+1,j}+x_{i+1,i+1}x_{i,j} \\
\lbrack i, i+1\mid i+1,j\rbrack_{g_BX}=\delta x_{i,i+1}x_{i+1,j}-x_{i+1,i+1}x_{i,j},
\end{eqnarray*}
where $\delta$ is $-1$ or $+1$ according to whether $(i+1,j)$ does or does not belong to $A$. Therefore 
\[x_{i+1,i+1}x_{i,j}\in I_2(g_AX)+I_2(g_BX),\]
which means that $I_2(g_AX)\neq I_2(g_BX)$. (Since it is a prime ideal, $I_2(g_AX)$ does not contain $x_{i+1,i+1}x_{i,j}$.)

Our next goal is to show that $\diam G(I)\leq \binom{m-1}{2}$. 
To prove this, take two subsets $A,B\subset \{(i,j):1\leq i<j\leq m\}$ such that $A\subset B$ and $B\setminus A=\{(i_0,j_0)\}$. We claim that 
\[\height(I_2(g_AX)+I_2(g_BX))=\height I+1=\binom{m}{2}+1.\]
In fact, it is easy to see that
\[I_2(g_AX)+I_2(g_BX)=I_2(g_AX)+ (x_{i_0,j_0} x_{ij}:\mbox{ both $i\neq i_0$ and $j\neq j_0$}).\]
Consider the ideal $I_2(g_AX)+I_2(g_BX)$ modulo $I_2(g_AX)$, so that we get the ideal
\[J= \overline{(x_{i_0,j_0} x_{ij}:\mbox{ both $i\neq i_0$ and $j\neq j_0$})}\subset R=S/I_2(g_AX).\]
By Krull's Hauptidealsatz, any minimal prime ideal $\pp$ of $\overline{(x_{i_0,j_0})}$ has height at most 1, and since $\pp\supseteq J$, it follows that $\height J\leq 1$. Because $R$ is a domain and $J$ is not the zero ideal, $\height J=1$. Thus the claim is proven.

Now, take two minimal prime ideals $\pp$ and $\qq$ of $I$. By what said before and the symmetry of the situation, we can assume that $\pp=I_2(X)$ and $\qq =I_2(g_AX)$ for a subset $A$ of $U=\{(i,j):1\leq i<j-1\leq m-1\}$. Pick a saturated chain $A_1\subset A_2 \subset ... \subset A_k=A$ such that $|A_i|=i$. Then, by what we proved above,
\[\height(I_2(X)+I_2(g_{A_1}X)) = \height(I_2(g_{A_{i-1}}X)+I_2(g_{A_i}X))=1 \ \ \ \forall \ i=2,\ldots ,k,\]
so $\diam G(I)\leq k \leq \binom{m-1}{2}$. In particular, $I$ is Hirsch.
\end{example}

\subsection{Cautionary examples and non-Hirsch ideals}

Let us finish with some examples. The first one is a caveat concerning the ``distance'' between two minimal primes. In the monomial case, if three minimal primes $\pp_1, \pp_2, \pp_3$ of a monomial ideal $I$ form a $2$-edge path in $G(I)$, then $\height(\pp_1 +\pp_3)$ is at most $2+ \height \pp_1$. Hence one is tempted to think that $\height(\pp_i+\pp_1)$ should somehow measure the graph-theoretical distance of $\pp_i$ from $\pp_1$. This is false for non-monomial ideals, as the following example (for $n\ge 4$) outlines.

\begin{example}
Let $S$ be the ring $\KK[x_1, \ldots, x_n, y_1, \ldots, y_n]$. 
Let $\pp_x$ (resp.\ $\pp_y$) be the prime ideal generated by $x_1, \ldots, x_{n-1}$ 
(resp.\ by $y_1, \ldots, y_{n-1}$). Clearly, 
\[\height \pp_x = n - 1 = \height \pp_y.\]
Next, consider the $2 \times n$ matrix with row vectors $(x_1, \ldots, x_n)$ and $(y_1, \ldots, y_n)$. Let $\pp$ be the prime ideal generated by the size-$2$ minors of such a matrix, and let 
\[I=\pp \cap \pp_x \cap \pp_y.\]
It is well known that $\height \pp=n-1$. Moreover, $\pp+\pp_x$ is contained in $(x_1,...,x_n)$, so it has height $n$. It follows that in $G(I)$ the primes $\pp$ and $\pp_x$ are connected by an edge. Symmetrically, there is an edge between $\pp$ and $\pp_y$. However, 
\[\height(\pp_x + \pp_y) \, = \, \height(x_1, \ldots, x_{n-1}, y_1, \ldots, y_{n-1}) \; = \; 2n-2.\] 
In conclusion, there is no upper bound for $\height(\pp_x + \pp_y)$,  even if $\pp_x$ and $\pp_y$ are two primes at distance $2$ in $G(I)$.
\end{example} 

\bigskip
Next, we highlight a construction (dual to taking products of polytopes) to obtain triangulated spheres whose Stanley--Reisner ring is ``far from being Hirsch''. Recall that if $P$ is any (convex) $(d+1)$-dimensional simplicial polytope with $n$ vertices, 
its polar dual $Q$ is a $(d+1)$-dimensional simple polytope with $n$ facets: The graph of $Q$ coincides with the dual graph of $\partial P$.
Moreover, the $k$-fold product $Q^k=Q \times \ldots \times Q$  is a $k(d+1)$-dimensional simple polytope with $kn$ facets. If the graph of $Q$ has diameter $\delta$, it is not difficult to show that the graph of $Q^k$ has diameter $k\delta$.

\begin{example}[Matschke--Santos--Weibel] Matschke, Santos and Weibel \cite{MSW} recently constructed a simplicial polytope $P$ with the following properties:
\begin{compactenum}[(i)]
\item The boundary $\Delta= \partial P$ of $P$ is a $19$-dimensional sphere with $40$ vertices;
\item the dual graph of $\Delta$ has diameter $21$.
\end{compactenum}
It follows that the ideal $I_\Delta \subset \KK[x_1, \ldots, x_{40}]$ has height $20$ and diameter $21$, so it is \emph{not} Hirsch. This is the smallest non-Hirsch sphere currently known. (The ideal $I_\Delta$ is monomial and radical, but it is not generated in degree two. Moreover, $S/I_\Delta$ is Gorenstein.) 

Let us apply the dual product construction sketched before to the $20$-dimensional polytope $P$ above. If $Q$ is the polar of $P$, let $\Delta_k$ denote the boundary of the polar dual of $Q^k$. By construction, $\Delta_k$ is a simplicial sphere with $40k$ vertices and 
dimension $20k-1$. Moreover, the dual graph of $\Delta_k$ is just the graph of $Q^k$, which has diameter $21k$. If $I_k \subset \KK[x_1, \ldots, x_{40k}]$ denotes the Stanley--Reisner ideal of $\Delta_k$, we have
\[\diam G(I_k)=21k \qquad \textrm{ and } \qquad \height(I_k) = 40k - (20k - 1) - 1= 20k.\] 
\end{example}

Very recently, Santos produced $d$-dimensional simplicial complexes $\Delta$  with $\diam G(I_{\Delta}) \in n^{\Theta(d)}$ \cite[Corollary 2.12]{Santos13}.
For Cohen--Macaulay $d$-complexes, however, the diameter of the dual graph is bounded above by $2^{d-2} n$, which for fixed $d$ is linear in $n$:


\begin{theorem}[{Larman \cite{Larman}, see also  
\cite[Theorems 3.12 and 3.14]{Santos13}}] 
Let $I \subset S=\KK[x_1,\ldots ,x_n]$ be a (squarefree) monomial ideal of height $c$.  
If $S/I$ is Cohen--Macaulay, \[\diam G(I) \le 2^{n-c-3} \, n. \] 
\end{theorem}

Our final example shows that even with the Cohen--Macaulay assumption, this type of upper bound (independent on the degree of generators) cannot  exist outside the world of monomial ideals. In fact, even if we prescribe $I$ to be a complete intersection, and even if we fix the parameters $\height(I)=2$ and $n=4$, the diameter of $G(I)$ can be arbitrarily high.  
\begin{example} \label{ex:unbounded}
If $\KK$ is algebraically closed, for any $N\in\N$, there are two polynomials $f,g \in S=\KK[x_1,\ldots ,x_4]$ such that $I=(f,g)$ is a complete intersection and $\diam G(I)=N$. To prove this, pick $N+2$ linear forms $\ell_1,\ldots ,\ell_{N+2}\in S$ such that any 4 of them are linearly independent, and set:
\[J = (\ell_1, \ell_2) \cap (\ell_2, \ell_3) \cap \ldots \cap (\ell_{N+1}, \ell_{N+2}).\] 
By construction, $J$ defines a connected union of lines in $\PP^3$ and $G(J)$ is a path on $N+1$ vertices. By a result of Mohan Kumar \cite[Theorem 2.15]{Ly}, $J$ is a set-theoretic complete intersection. In other words, there exist $2$ polynomials $f$ and $g$ such that the ideals $I=(f,g)$ and $J$ have the same radical, namely, $J$. In particular, $G(I) = G(J)$  and $\diam G(I) =N$. Note that 

\[
\begin{array}{ll}
\deg f + \deg g \, &= \, \reg (S/I), \quad \textrm{ and} \\ 
\deg f \,\cdot\, \deg g \, &= \, e(S/I) \, \ge \, e(S/J) \, = \, N.
\end{array}
\]

It follows that $\reg(S/I) \ \ge \ 2\sqrt{N}$. So if $N$ is very large, the regularity of $S/I$ is also large. In contrast, the graph $G(I) = G(J)$ is not even $2$-connected. There is however no contradiction with Main Theorem \ref{mthm:1}. In fact, $S/I$ is Gorenstein, but $I$ does not define a subspace arrangement; whereas $\sqrt{I}$ defines a subspace arrangement, but $S/\sqrt{I}$ is not Gorenstein. 
\end{example}

By Proposition \ref{prop:easy}, the phenomenon of Example \ref{ex:unbounded} cannot appear in a polynomial ring $S$ with less than 4 variables. 

\bigskip

{\bf Acknowledgments}. We thank David Eisenbud for suggesting us to look at his paper with Bayer \cite{BE}, which inspired the results of Section \ref{sec:arrangement}. Thanks to Christos Athanasiadis, for pointing us the reference \cite{Klee}. Finally we would like to thank Aldo Conca for useful discussions and for proofreading part of the paper.

\bigskip



\begin{thebibliography}{BMS2}
\small
\itemsep=-1.4mm

\bibitem[AB13]{AB}
K. Adiprasito, B. Benedetti, \emph{The Hirsch conjecture holds for normal flag complexes}, to appear in Math. Oper. Research. Preprint at 
\url{arxiv.org/abs/1303.3598}, 2013.

\bibitem[Ath09]{Athanasiadis}
Ch. A. Athanasiadis, \emph{On the graph connectivity of skeleta of convex polytopes}. Discr. Comput. Geom. 42, pp. 155--165, 2009. 

\bibitem[Bar82]{Barnette}
D. Barnette, \emph{Decomposition of homology manifolds and their graph}, Isr. J. Math. 41, pp. 203--212, 1982.

\bibitem[BE91]{BE}
D. Bayer, D. Eisenbud, \emph{Graph curves. With an appendix by Sung Won Park}, Adv. Math. 86, pp. 1-40, 1991.

\bibitem[BPS05]{BPS}
A. Bj\"orner, I. Peeva, J. Sidman, \emph{Subspace arrangements defined  by products of linear forms} J. Lond. Math. Soc. 71, pp. 273--288, 2005.

\bibitem[BV13]{BjVo}
A. Bj\"orner, K. Vorwerk, \emph{On the connectivity of manifold graphs}, to appear in Proc. Amer. Math. Soc. 
Preprint at \url{arxiv.org/abs/1207.5381}, 2013.


\bibitem[Bol98]{Bollobas}
B. Bollob\'{a}s, \emph{Modern Graph Theory}. Graduate Texts in Mathematics 184, Springer, 1998.

\bibitem[BH93]{BrunsHerzog}
W. Bruns, J. Herzog \emph{Cohen--Macaulay Rings}. Cambridge University Press, 1993.

\bibitem[DS02]{DS}
H. Derksen, J. Sidman, \emph{A sharp bound for the Castelnuovo-Mumford regularity of subspace arrangements}, Adv. Math. 172, pp. 151--157, 2002.

\bibitem[Die05]{Diestel}
H. Diestel, \emph{Graph Theory}. Springer, 3rd ed., 2005.

\bibitem[Eis05]{Ei}
D. Eisenbud, \emph{The Geometry of Syzygies}, Graduate Texts in Mathematics 229, Springer 2005.

\bibitem[M2]{macaulay2}
D. R. Grayson, M. E.  Stillman, \emph{Macaulay2, a software system for research
                   in algebraic geometry}. Available at \url{www.math.uiuc.edu/Macaulay2/}.

\bibitem[Har77]{Hartshorne}
R. Hartshorne, \emph{Algebraic Geometry}. Springer, Graduate Texts in Mathematics 52, 1977.

\bibitem[Har62]{Ha}
R. Hartshorne, \emph{Complete intersection and connectedness}, Amer. J. Math. 84, pp. 497-508, 1962.

\bibitem[Har79]{Ha1}
R. Hartshorne, \emph{Complete Intersections in Characteristic $p>0$}, Amer. J. Math. 101, pp. 380-383, 1979.

\bibitem[HTT05]{HTT}
J. Herzog, Y. Takayama, N. Terai, \emph{On the radical of a monomial ideal}, Arch. der Math. 85, pp. 
397-408, 2005.

\bibitem[KK92]{KaKl}
G. Kalai, D. J. Kleitman, \emph{A quasi-polynomial bound for the diameter of graphs of polyhedra}, Bull. Amer. Math. Soc., pp. 315--316, 1992.

\bibitem[Kal10]{Kpolymath}
G. Kalai (coordinator), \emph{Polymath 3: Polynomial Hirsch Conjecture}, online platform, September-October 2010. At \url{gilkalai.wordpress.com/2010/09/29/polymath-3-polynomial-hirsch-conjecture},




\bibitem[Kle75]{Klee}
V. Klee, \emph{A $d$-pseudomanifold with $f_0$ vertices has at least $df_0 - (d-1)(d+2)$ $d$-simplices}, Houston J. Math. 1, 1975.

\bibitem[Lar70]{Larman}
D. G. Larman, \emph{Paths on polytopes}, Proc. London Math. Soc. 20:3, pp. 161–178, 1970. 

\bibitem[Lyu89]{Ly}
G. Lyubeznik, \emph{A survey of problems and results on the number of defining equations}, MSRI Publications 15, pp. 375--390, 1989.

\bibitem[MSW13]{MSW} B. Matschke, F. Santos, C. Weibel, \emph{The width of 5-dimensional prismatoids}. 
Preprint at \url{arxiv.org/abs/1202.4701 }, 2013.

\bibitem[MS05]{MiSt} E. Miller, B. Sturmfels, \emph{Combinatorial Commutative Algebra}, Graduate Texts in Mathematics 227, Springer, 2005.


\bibitem[Mig98]{Mi}
J. C. Migliore, \emph{Introduction to Liaison Theory and Deficiency Modules}, Progress in Mathematics 165, Birkh\" auser Boston, 1998.


\bibitem[San12]{Santos}
F. Santos, \emph{A counterexample to the Hirsch conjecture}, Ann. Math. (2) 176, pp. 383--412, 2012.

\bibitem[San13]{Santos13}
F. Santos, \emph{Recent progress on the combinatorial diameter of polytopes and simplicial complexes}, TOP, Volume 21, Issue 3, pp. 426-460, 2013.


\bibitem[Sch82]{Sc}
P. Schenzel, \emph{Notes on liaison and duality}, J. Math. Kyoto Univ. 22, pp. 485-498, 1982.

\bibitem[Wot09]{Wotzlaw}
R. F. Wotzlaw, \emph{Incidence Graphs and Unneighborly Polytopes}. PhD Thesis, TU Berlin, 2009. Available online at \url{http://opus4.kobv.de/opus4-tuberlin/frontdoor/index/index/docId/2116}




\end{thebibliography}
\end{document}